\documentclass[a4paper,12pt]{amsart}
\usepackage{amssymb}
\RequirePackage{calrsfs}

\DeclareSymbolFont{rsfscript}{OMS}{rsfs}{m}{b}
\DeclareSymbolFontAlphabet{\mathrsfs}{rsfscript}
\renewcommand{\mathcal}{\mathrsfs}

\def\bfit{\bfseries\itshape}

\input xypic
\xyoption{all}
\xyoption{arc}

\newtheorem{theo}{Th\'eor\`eme}[section]

\newtheorem{prop}[theo]{Proposition}

\newtheorem{lem}[theo]{Lemme}

\newtheorem{coro}[theo]{Corollaire}

\def\example#1{{\refstepcounter{theo}\label{#1}\noindent\sc Exemple 
\arabic{section}.\arabic{theo} - }}

\def\equat{\refstepcounter{theo}$$~}
\def\endequat{\leqno{\boldsymbol{(\arabic{section}.\arabic{theo})}}~$$}

\def\AG{{\mathfrak A}}

    \def\QM{{\mathbb{Q}}}
    \def\RM{{\mathbb{R}}}

    \def\ZM{{\mathbb{Z}}}

    \def\CC{{\mathcal{C}}}
    \def\DC{{\mathcal{D}}}
    \def\EC{{\mathcal{E}}}
    \def\FC{{\mathcal{F}}}
    
    \def\HC{{\mathcal{H}}}

    \def\LC{{\mathcal{L}}}

    \def\OC{{\mathcal{O}}}
    \def\PC{{\mathcal{P}}}

    \def\SC{{\mathcal{S}}}
    
    \def\UC{{\mathcal{U}}}
    \def\VC{{\mathcal{V}}}

\def\Fti{{\tilde{F}}}

\def\g{\gamma}
\def\G{\Gamma}

\def\e{\varepsilon}
\def\ph{\varphi}
\def\l{\lambda}
\def\L{\Lambda}

\def\s{\sigma}

\def\t{\tau}

\def\phb{{\boldsymbol{\varphi}}}

          \def\psit{{\tilde{\psi}}}

\def\phba{{\bar{\varphi}}}

\def\psiba{{\bar{\psi}}}

\def\sigba{{\bar{\s}}}

\DeclareMathOperator{\Id}{{\mathrm{Id}}}

\DeclareMathOperator{\Ker}{{\mathrm{Ker}}}

\DeclareMathOperator{\Pos}{{\mathrm{Pos}}}
\DeclareMathOperator{\Posbar}{{\overline{\mathrm{Pos}}}}

\def\imp{\Rightarrow}

\def\to{\rightarrow}
\def\longto{\longrightarrow}

\def\fonction#1#2#3#4#5{\begin{array}{rccc}
{#1} : & {#2} & \longto & {#3} \\
& {#4} & \longmapsto & {#5} 
\end{array}}

\def\fonctio#1#2#3#4{\begin{array}{ccc}
{#1} & \longto & {#2} \\
{#3} & \longmapsto & {#4} 
\end{array}}

\def\vide{\varnothing}

\def\DS{\displaystyle}
\def\SS{\scriptstyle}

\def\finl{~$\SS \square$}

\def\infspe{\hspace{0.1em}\mathop{\preccurlyeq}\nolimits\hspace{0.1em}}

\def\lexp#1#2{\kern\scriptspace\vphantom{#2}^{#1}\kern-\scriptspace#2}
\def\le{\mathop{\leqslant}\nolimits}
\def\ge{\mathop{\geqslant}\nolimits}

\mathchardef\inferieur="321E
\mathchardef\superieur="321F

\def\eqna{\begin{eqnarray*}}
\def\endeqna{\end{eqnarray*}}

\def\tors{tores maximaux }

\def\tors{{\mathrm{tors}}}

\def\itemth#1{\item[${\mathrm{(#1)}}$]}

\catcode`\@=11
\long\def\@car#1#2\@nil{#1}
\long\def\@first#1#2{#1}
\long\def\@second#1#2{#2}
\long\def\ifempty#1{\expandafter\ifx\@car#1@\@nil @\@empty
  \expandafter\@first\else\expandafter\@second\fi}
\catcode`\@=12

\def\positif{\PC\! os}

\DeclareMathOperator{\can}{{\mathrm{can}}}

\DeclareMathOperator{\codim}{{\mathrm{codim}}}

\def\dotcup{\hskip1mm\dot{\cup}\hskip1mm}

\begin{document}

\baselineskip=16pt

\title{Topologie sur l'ensemble des parties \\ positives d'un r\'eseau}

\author{C\'edric Bonnaf\'e}
\address{\noindent 
Labo. de Math. de Besan\c{c}on (CNRS: UMR 6623), 
Universit\'e de Franche-Comt\'e, 16 Route de Gray, 25030 Besan\c{c}on
Cedex, France} 

\makeatletter
\email{cedric.bonnafe@univ-fcomte.fr}
\urladdr{http://www-math.univ-fcomte.fr/pp\_Annu/CBONNAFE/}

\thanks{L'auteur est en partie financ\'e par l'ANR (Projet 
No JC07-192339)}


\date{\today}
\def\abstractname{R\'esum\'e}

\begin{abstract} 
Nous d\'efinissons une notion de {\it partie positive} 
d'un r\'eseau $\Lambda$ et munissons l'ensemble de ces parties 
positives d'une topologie. Nous \'etudions ensuite 
quelques propri\'et\'es de cette topologie, en la comparant 
\`a celle de $V^*/\RM_{> 0}$, o\`u $V^*$ est l'espace vectoriel dual de 
$\RM \otimes_\ZM \Lambda$.

\bigskip

\noindent{\sc Abstract.} We define a notion of {\it positive part} 
of a lattice $\Lambda$ and we endow the set of such positive parts 
with a topology. We then study some properties of this topology, 
by comparing it with the one of $V^*/\RM_{> 0}$, where $V^*$ is 
the dual vector space of $\RM \otimes_\ZM \Lambda$.
\end{abstract}

\maketitle

\pagestyle{myheadings}

\markboth{\sc C. Bonnaf\'e}{\sc Topologie sur l'ensemble 
des parties positives d'un r\'eseau}

Fixons dans cet article un r\'eseau $\L$ et notons 
$V=\RM \otimes_\ZM \L$. Une partie $X$ de $\L$ est dite 
{\it positive} si les conditions 
suivantes sont satisfaites~:
\begin{itemize}\itemindent1cm
\item $\L=X \cup (-X)$.

\item $X+X \subset X$.

\item $X \cap (-X)$ est un sous-groupe de $\L$.
\end{itemize}
Nous noterons $\positif(\L)$ l'ensemble des parties 
positives de $\L$. Le but de cet article est d'\'etudier 
cet ensemble, et notamment de le munir d'une topologie, 
ceci de fa\c{c}on fonctorielle. 

Nous d\'efinissons notamment deux applications 
$\Posbar : V^*/\RM_{>0} \to \positif(\L)$ et 
$\pi : \positif(\L) \to V^*/\RM_{>0}$ dont nous montrons 
qu'elles sont continues 
et v\'erifient $\pi \circ \Posbar = \Id_{V^*/\RM_{>0}}$. 
De plus, $\Posbar$ induit un hom\'emorphisme sur son image, et 
celle-ci est dense dans $\positif(\L)$.  Ces propri\'et\'es 
sont r\'esum\'ees dans le th\'eor\`eme \ref{theo:topologie}. 

Nous terminons cet article par une \'etude des {\it arrangement d'hyperplans} 
dans $\positif(\L)$, en montrant notamment que l'intuition classique 
sur les arrangements d'hyperplans r\'eels reste valide. 

Ces r\'esultats seront utilis\'es ult\'erieurement par l'auteur 
pour \'etudier le comportement des cellules de Kazhdan-Lusztig 
lorsque les param\`etres varient \cite{bonnafe}.

\newpage

\tableofcontents

\bigskip

\section{Parties positives d'un r\'eseau\label{section positive}}

\medskip

Le but de cette section est d'\'etudier l'ensemble des parties positives 
de $\L$. Nous munirons cet ensemble d'une topologie dans la section suivante. 

\def\tors{{\mathrm{tor}}}

\bigskip

\subsection{D\'efinitions, pr\'eliminaires\label{sous positive}} 
Une partie $X$ de $\L$ est dite {\it positive} si les conditions 
suivantes sont satisfaites~:
\begin{itemize}\itemindent1cm
\itemth{P1} $\L=X \cup (-X)$.

\itemth{P2} $X+X \subset X$.

\itemth{P3} $X \cap (-X)$ est un sous-groupe de $\L$.
\end{itemize}
Nous noterons $\positif(\L)$ l'ensemble des parties 
positives de $\L$. Donnons quelques exemples. 
Pour commencer, notons que $\L \in \positif(\L)$. 
Soit $\G$ un groupe totalement ordonn\'e et soit $\ph : \L \to \G$ 
un morphisme de groupes. Posons 
$$\Pos(\ph)=\{\l \in \L~|~\ph(\l) \ge 0\}$$
$$\Pos^+(\ph)=\{\l \in \L~|~\ph(\l) > 0\}.\leqno{\text{et}}$$
Alors il est clair que 
\equat\label{noyau pos}
\Ker \ph = \Pos(\ph) \cap \Pos(-\ph) = \Pos(\ph) \cap -\Pos(\ph)
\endequat
et que 
\equat\label{positif exemple}
\text{\it $\Pos(\ph)$ est une partie positive de $\L$}.
\endequat

\bigskip

\begin{lem}\label{proprietes positives}
Soit $X$ une partie positive de $\L$. Alors~:
\begin{itemize}
\itemth{a} $-X \in \positif(\L)$. 

\itemth{b} $0 \in X$.

\itemth{c} Si $\l \in \L$ et si $r \in \ZM_{>0}$ est tel que 
$r\l \in X$. Alors $\l \in X$.

\itemth{d} $\L/(X \cap (-X))$ est sans torsion. 
\end{itemize}
\end{lem}

\bigskip

\begin{proof}
(a) est imm\'ediat. 
(b) d\'ecoule de la propri\'et\'e (P1) des parties positives. 
(d) d\'ecoule de (c). Il nous reste \`a montrer (c). 
Soient $\l \in \L$ et $r \in \ZM_{>0}$ tels que $r\l \in X$. 
Si $\l \not\in X$, alors $-\l \in X$ 
d'apr\`es la propri\'et\'e (P1). 
D'o\`u $\l=r\l + (r-1)(-\l) \in X$ d'apr\`es (P2), ce qui est contraire 
\`a l'hypoth\`ese. Donc $\l \in X$. 
\end{proof}

\bigskip

%\begin{coro}\label{torsion incluse}
%Si $X \in \PC(\L)$, alors $\L_\tors \subset X$.
%\end{coro}
%
%\bigskip

Nous allons montrer une forme de r\'eciproque facile \`a la propri\'et\'e 
\ref{positif exemple}. Soit $X \in \positif(\L)$. Notons 
$\can_X : \L \to \L/(X \cap(-X))$ le morphisme canonique. Si $\g$ et $\g'$ 
appartiennent \`a $\L/(X \cap(-X))$, nous \'ecrirons $\g \le_X \g'$ s'il existe 
un repr\'esentant de $\g' - \g$ appartenant \`a $X$. Il est facile 
de v\'erifier que
\equat\label{relation independante}
\text{\it $\g \le_X \g'$ si et seulement si tout repr\'esentant 
de $\g'-\g$ appartient \`a $X$.}
\endequat
On d\'eduit alors facilement des propri\'et\'es (P1), (P2) et (P3) 
des parties positives que 
\equat\label{ordre X}
\text{\it $(\L/(X \cap (-X)),\le_X)$ est un groupe ab\'elien 
totalement ordonn\'e}
\endequat
et que 
\equat\label{X pos}
X=\Pos(\can_X).
\endequat

\bigskip

\subsection{Cons\'equences du th\'eor\`eme de Hahn-Banach} 
Si $X$ est une partie positive de $\L$, on pose 
$X^+=X \setminus (-X)$. On a alors 
\equat\label{disjonction}
\L = X \dotcup (-X^+) = X^+ \dotcup (-X) 
= X^+ \dotcup (X \cap (-X)) \dotcup (-X^+),
\endequat
o\`u $\dot{\cup}$ d\'esigne l'union disjointe. 
De plus, si $\ph : \L \to \G$ est un morphisme de groupes 
ab\'eliens et si $\G$ est un groupe totalement ordonn\'e, alors 
\equat\label{pos plus}
\Pos^+(\ph) = \Pos(\ph)^+.
\endequat
Si $\ph$ est une forme lin\'eaire sur $V$, nous noterons abusivement 
$\Pos(\ph)$ et $\Pos^+(\ph)$ les parties $\Pos(\ph|_\L)$ et 
$\Pos^+(\ph|_\L)$ de $\L$. 

\bigskip

\begin{lem}\label{equivalent separant}
Soit $X$ une partie positive propre de $\L$, soit $\G$ un groupe ab\'elien 
totalement ordonn\'e archim\'edien 
et soit $\ph : \L \to \G$ un morphisme de groupes. 
Alors les conditions suivantes sont \'equivalentes~:
\begin{itemize}
\itemth{1} $X \subseteq \Pos(\ph)$;

\itemth{2} $X^+ \subseteq \Pos(\ph)$;

\itemth{3} $\Pos^+(\ph) \subseteq X^+$;

\itemth{4} $\Pos^+(\ph) \subseteq X$.
\end{itemize}
\end{lem}

\begin{proof}
Il est clair que (1) implique (2) et que (3) implique (4). 

\medskip

Montrons que (2) implique (3). Supposons 
donc que (2) est v\'erifi\'ee. Soit $\l \in \L$ tel que 
$\ph(\l) > 0$ et supposons que $\l \not\in X^+$. 
Alors, d'apr\`es \ref{disjonction}, $\l \in -X$. 
Or, si $\mu \in \L$, il existe $k \in \ZM_{> 0}$ tel que 
$\ph(\mu - k \l)=\ph(\mu)-k\ph(\l) < 0$ (car $\G$ est archim\'edien). Donc $\mu-k\l \not\in X^+$ 
d'apr\`es (2). Donc $\mu - k \l \in -X$ d'apr\`es \ref{disjonction}. 
Donc, $\mu=(\mu -k\l) + k\l \in (-X)$. Donc $\L \subseteq -X$, 
ce qui est contraire \`a l'hypoth\`ese.

\medskip

Montrons que (4) implique (1). 
Supposons donc que $\Pos^+(\ph) \subseteq X$. 
En prenant le compl\'ementaire dans $\L$, 
on obtient $(-X^+)=(-X)^+ \subseteq \Pos(-\ph)$ et donc, 
puisque (2) implique (3), on a $\Pos^+(-\ph) \subseteq -X^+$. 
En reprenant le compl\'ementaire dans $\L$, on 
obtient $X \subseteq \Pos(\ph)$. 
\end{proof}

\bigskip

Nous aurons aussi besoin du lemme suivant~:

\bigskip

\begin{lem}\label{solutions}
Soient $\l_1$,\dots, $\l_n$ des \'el\'ements de $\L$ et supposons 
trouv\'e un $n$-uplet $t_1$,\dots, $t_n$ de nombres r\'eels 
{\bfit strictement} positifs tels que $t_1 \l_1 + \cdots + t_n \l_n=0$. 
Alors il existe des entiers naturels non nuls $r_1$,\dots, $r_n$ tels que 
$r_1 \l_1 + \cdots + r_n \l_n = 0$.
\end{lem}

\bigskip

\begin{proof}
Notons $\SC$ l'ensemble des $n$-uplets $(u_1,\dots,u_n)$ 
de nombres r\'eels qui satisfont 
$$\begin{cases}
u_1+\cdots + u_n = 1, &\\
u_1 \l_1 + \cdots + u_n \l_n = 0.&\\
\end{cases}$$
\'Ecrit dans une base de $\L$, ceci est un syst\`eme lin\'eaire 
d'\'equations \`a coefficients dans $\QM$. Le proc\'ed\'e d'\'elimination de 
Gauss montre que l'existence d'une solution {\it r\'eelle} implique 
l'existence d'une solution {\it rationnelle} 
$t^\circ = (t_1^\circ,\dots, t_n^\circ)$ 
et l'existence de vecteurs $v_1$,\dots, $v_r \in \QM^n$ tels que 
$$\SC=\{t^\circ + x_1 v_1 + \cdots +x_r v_r~|~(x_1,\dots,x_r) \in \RM^r\}.$$
En particulier, il existe $x_1$,\dots, $x_r \in \RM$ tels que 
$$(t_1,\dots,t_n) = t^\circ + x_1 v_1 + \cdots +x_r v_r.$$
Puisque $t_1$,\dots, $t_n$ sont strictement positifs, 
il existe $x_1'$,\dots, $x_r'$ dans $\QM$ tels que les 
coordonn\'ees de $t^\circ + x_1' v_1 + \cdots +x_r' v_r$ soient 
strictement positives. Posons alors 
$(u_1,\dots,u_n)=t^\circ + x_1' v_1 + \cdots +x_r' v_r$. 
On a donc $u_i \in \QM_{> 0}$ pour tout $i$ et 
$$u_1 \l_1 + \cdots + u_n \l_n = 0.$$
Quitte \`a multiplier par le produits des d\'enominateurs des $u_i$, 
on a trouv\'e $r_1$,\dots, $r_n \in \ZM_{> 0}$ tels que 
$$r_1 \l_1 + \cdots + r_n \l_n = 0,$$
comme attendu.
\end{proof}

\bigskip

\begin{theo}\label{hahn banach}
Soit $X$ une partie positive de $\L$ diff\'erente de $\L$. Alors~:
\begin{itemize}
\itemth{a} Il existe une forme lin\'eaire $\ph$ sur $V$ telle que 
$X \subseteq \Pos(\ph)$. 

\itemth{b} Si $\ph$ et $\ph'$ sont deux formes lin\'eaires 
sur $V$ telles que $X \subseteq \Pos(\ph) \cap \Pos(\ph')$, 
alors il existe $\kappa \in \RM_{> 0}$ 
tel que $\ph'=\kappa \ph$.
\end{itemize}
\end{theo}

\bigskip

\begin{proof}
Notons $\CC^+$ l'enveloppe convexe de $X^+$. 
Notons que $X^+$ (et donc $\CC^+$) 
est non vide car $X \neq \L$. Nous allons commencer 
par montrer que $0 \not\in \CC^+$. Supposons donc que $0 \in \CC^+$. 
Il existe donc $\l_1$,\dots, $\l_n$ dans $X^+$ et $t_1$,\dots, $t_n$ 
dans $\RM_{> 0}$ tels que 
$$\begin{cases}
t_1+\cdots + t_n = 1, &\\
t_1 \l_1 + \cdots + t_n \l_n = 0.&\\
\end{cases}$$
D'apr\`es le lemme \ref{solutions}, il existe $r_1$,\dots, $r_n \in \ZM_{> 0}$ 
tels que $m_1 \l_1 = - (m_2 \l_2 + \cdots + m_n \l_n)$. Donc 
$m_1\l_1 \in X \cap -X$ 
(voir la propri\'et\'e (P2)). Donc, d'apr\`es le lemme 
\ref{proprietes positives} (a) et (c), on a $\l_1 \in X \cap -X$, 
ce qui est impossible car $\l_1 \in X^+ = X \setminus (-X)$. 
Cela montre donc que 
$$0 \not\in \CC^+.$$
L'ensemble $\CC^+$ \'etant convexe, il d\'ecoule du th\'eor\`eme de 
Hahn-Banach qu'il existe une forme lin\'eaire non nulle $\ph$ 
sur $V$ telle que
$$\CC^+ \subseteq \{\l \in V~|~\ph(\l) \ge 0\}.$$
En particulier, 
$$X^+ \subseteq \CC^+ \cap \L \subseteq \Pos(\ph).\leqno{(*)}$$
D'apr\`es le lemme \ref{equivalent separant}, et puisque 
$\RM$ est archim\'edien, on a bien $X \subseteq \Pos(\ph)$. 
Cela montre (a). 

\medskip

Soit $\ph'$ une autre forme lin\'eaire telle que $X \subseteq \Pos(\ph')$. 
Posons $U=\{\l \in V~|~\ph(\l) > 0$ et $\ph'(\l) < 0\}$. 
Alors $U$ est un ouvert de $V$. Fixons une $\ZM$-base $(e_1,\dots,e_d)$ 
de $\L$. Si $U \neq \vide$, alors il existe 
$\l \in \QM \otimes_\ZM \L$ et $\e \in \QM_{> 0}$ tels que 
$\l$, $\l + \e e_1$,\dots, $\l+\e e_d$ appartiennent \`a $U$. 
Quitte \`a multiplier par le produit des d\'enominateurs de 
$\e$ et des coordonn\'ees de $\l$ dans la base $(e_1,\dots,e_d)$, 
on peut supposer que $\l \in \L$ et $\e \in \ZM_{> 0}$. Mais 
il est clair que $U \cap X^+=\vide$ et $U \cap (-X^+) = \vide$. 
Donc $U \cap \L$ est contenu dans $X \cap (-X)$. Ce dernier 
\'etant un sous-groupe, on en d\'eduit que $\e e_i \in X \cap (-X)$ 
pour tout $i$. D'o\`u, d'apr\`es le lemme \ref{proprietes positives} (c), 
$e_i \in X \cap (-X)$ pour tout $i$. Donc $X=\L$, ce qui 
est contraire \`a l'hypoth\`ese. On en d\'eduit que $U$ est 
vide, c'est-\`a-dire qu'il existe $\kappa \in \RM_{> 0}$ tel que 
$\ph'=\kappa\ph$. D'o\`u (b).
\end{proof}

\bigskip

Si $\ph$ est une forme lin\'eaire sur $V$, nous noterons 
$\phba$ sa classe dans $V^*/\RM_{>0}$. Nous noterons $p : V^* \longto V^*/\RM_{>0}$
la projection canonique. L'application $\Pos : V^* \to \positif(\L)$ 
se factorise \`a travers $p$ en une application $\Posbar : 
V^*/\RM_{>0} \to \positif(\L)$ 
rendant le diagramme
$$\diagram
V^* \ddrrto^{\DS{\Pos}} \ddto_{\DS{p}}&& \\
&&\\
V^*\!/\RM_{>0} \rrto^{\DS{\Posbar}} && \positif(\L)
\enddiagram$$
commutatif. D'autre part, 
si $X \in \positif(\L)$, nous noterons $\pi(X)$ l'unique 
\'el\'ement $\phba \in V^*/\RM_{>0}$ tel que $X \subseteq \Posbar(\phba)$ 
(voir le th\'eor\`eme \ref{hahn banach}). On a donc d\'efini 
deux applications
$$\diagram
V^*/\RM_{>0}\rrto^{\DS{\Posbar}} && 
\positif(\L) \rrto^{\DS{\pi}} && V^*\!/\RM_{>0} 
\enddiagram$$
et le th\'eor\`eme \ref{hahn banach} (b) montre que
\equat\label{pi pos}
\pi \circ \Posbar = \Id_{V^*\!/\RM_{>0}}.
\endequat
Donc $\pi$ est surjective et $\Posbar$ est injective. 
En revanche, ni $\pi$, ni $\Posbar$ ne sont des bijections (sauf 
si $\L$ est de rang $1$). Nous allons d\'ecrire les fibres de $\pi$~:

\bigskip 

\begin{prop}\label{fibres pi}
Soit $\ph$ une forme lin\'eaire non nulle sur $V$. Alors 
l'application 
$$\fonction{i_{\phba}}{\positif(\Ker \ph|_\L)}{\pi^{-1}(\phba)}{X}{X 
\cup \Pos^+(\ph)}$$
est bien d\'efinie et bijective. Sa r\'eciproque est l'application 
$$\fonctio{\pi^{-1}(\phba)}{\positif(\Ker \ph|_\L)}{Y}{Y \cap \Ker \ph|\l.}$$
\end{prop}

\noindent{\sc Remarque - } 
Il est facile de voir que $\pi^{-1}(\bar{0})=\{\L\}$. Nous 
pouvons aussi d\'efinir une application 
$i_{\bar{0}} : \positif(\L) \to \positif(\L)$ 
par la m\^eme formule que dans la proposition 
\ref{fibres pi}~: alors $i_{\bar{0}}$ est tout simplement 
l'application identit\'e mais on a dans ce cas-l\`a 
$\pi^{-1}(\bar{0})\neq i_{\bar{0}}(\positif(\L))$.\finl

\bigskip

\begin{proof}
Montrons tout d'abord que l'application $i_\phba$ est bien d\'efinie. 
Soit $X \in \positif(\Ker \ph|\L)$. Posons $Y=X \cup \Pos^+(\ph)$. 
Montrons que $Y$ est une partie positive de $\L$. Avant cela, 
notons que 
$$Y \subseteq \Pos(\ph).\leqno{(*)}$$

(1) Si $\l \in \L$, deux cas se pr\'esentent. Si $\ph(\l) \neq 0$, 
alors $\l \in \Pos^+(\ph) \cup -\Pos^+(\ph) \subseteq Y \cup (-Y)$. 
Si $\ph(\l)=0$, alors $\l \in \Ker \ph|_\L$, donc 
$\l \in X \cup (-X) \subseteq Y \cup (-Y)$ car $X$ est 
une partie positive de $\Ker \ph|_\L$. Donc $\L=Y \cup (-Y)$. 

\smallskip

(2) Soient $\l$, $\mu \in Y$. Montrons que $\l+\mu \in Y$. 
Si $\ph(\l+\mu) > 0$, alors $\l+\mu \in \Pos^+(\ph) \subseteq Y$. 
Si $\ph(\l+\mu)=0$, alors il r\'esulte de $(*)$ que 
$\ph(\l)=\ph(\mu)=0$, donc $\l$, $\mu \in \Ker \ph|_\L$. 
En particulier, $\l$, $\mu \in X$ et donc 
$\l+\mu \in X+X \subseteq X \subseteq Y$. 
Donc $Y + Y \subseteq Y$.

\smallskip

(3) On a $Y \cap (-Y)=X \cap (-X)$, donc $Y \cap (-Y)$ est un sous-groupe 
de $\L$. 

\medskip

Les points (1), (2) et (3) ci-dessus montrent que $Y$ est une partie 
positive de $\L$. L'inclusion $(*)$ montre que $\pi(Y)=\phba$, 
c'est-\`a-dire que $Y \in \pi^{-1}(\phba)$. Donc l'application 
$i_\phba$ est bien d\'efinie.

\medskip

Elle est injective car, si $X \in \positif(\Ker \ph|_\L)$, alors 
$X \cap \Pos^+(\ph) = \vide$. Montrons maintenant 
qu'elle est surjective. Soit $Y \in \pi^{-1}(\phba)$. 
Posons $X=Y \cap \Ker \ph|_\L$. Alors $X$ est une partie 
positive de $\Ker \ph|_\L$ d'apr\`es le corollaire 
\ref{pos inclusion}. Posons $Y'=X \cup \Pos^+(\ph)$. 
Il nous reste \`a montrer que $Y=Y'$. 

Tout d'abord, $\Pos^+(\ph) \subseteq Y$ car $\pi(Y)=\phba$ 
par hypoth\`ese 
et $X \subseteq Y$. Donc $Y' \subseteq Y$. R\'eciproquement, 
si $\l \in Y$, deux cas se pr\'esentent. Si $\ph(\l) > 0$, 
alors $\l \in \Pos^+(\ph) \subseteq Y'$. Si $\ph(\l)=0$, 
alors $\l \in Y \cap \Ker \ph|_\L=X \subseteq Y'$. 
Dans tous les cas, $\l \in Y'$.
\end{proof}

\bigskip

\example{maximal} 
Si $\ph$ est une forme lin\'eaire non nulle sur $V$, alors 
$\Pos(\ph) = i_{\phba}(\Ker \ph|_\L)$.\finl

\bigskip

Nous pouvons maintenant classifier les parties positives de 
$\L$ en termes de formes lin\'eaires. Notons $\FC(\L)$ 
l'ensemble des suites finies $(\ph_1,\dots,\ph_r)$ 
telles que (en posant $\ph_0=0$), pour tout $i \in \{1,2,\dots,r\}$, 
$\ph_i$ soit une forme lin\'eaire non nulle sur 
$\RM \otimes_\ZM (\L \cap \Ker \ph_{i-1})$. 
Par convention, nous supposerons que la suite vide, not\'ee $\vide$, 
appartient \`a $\FC(\L)$. 

Posons $d = \dim V$. 
Notons que, si $(\ph_1,\dots,\ph_r) \in \FC(\L)$, alors $r \le d$. 
Nous d\'efinissons donc l'action suivante de $(\RM_{>0})^d$ sur $\FC(\L)$~: 
si $(\kappa_1,\dots,\kappa_d) \in (\RM_{>0})^d$ et si 
$(\ph_1,\dots,\ph_r) \in \FC(\L)$, on pose 
$$(\kappa_1,\dots,\kappa_d)\cdot (\ph_1,\dots,\ph_r) = 
(\kappa_1 \ph_1,\dots,\kappa_r \ph_r).$$
Munissons $\RM^r$ de l'ordre lexicographique~: c'est un groupe ab\'elien 
totalement ordonn\'e et $(\ph_1,\dots,\ph_r) : \L \to \RM^r$ 
est un morphisme de groupes. Donc $\Pos(\ph_1,\dots,\ph_r)$ est 
bien d\'efini et appartient \`a $\positif(\L)$. En fait, toute 
partie positive de $\L$ peut \^etre retrouv\'ee ainsi~:

\bigskip

\begin{prop}\label{surjection pos}
L'application 
$$\fonctio{\FC(\L)}{\positif(\L)}{\phb}{\Pos(\phb)}$$
est bien d\'efinie et induit une bijection 
$\FC(\L)/(\RM_{> 0})^d \stackrel{\sim}{\longrightarrow} \positif(\L)$. 
\end{prop}

\noindent{\sc Remarque - } Dans cette proposition, on a pos\'e 
par convention $\Pos(\vide)=\L$.\finl

\bigskip

\begin{proof}
Cela r\'esulte imm\'ediatement d'un raisonnement par r\'ecurrence 
sur le rang de $\L$ en utilisant le th\'eor\`eme \ref{hahn banach} 
et la proposition \ref{fibres pi}.
\end{proof}

\bigskip

Si $\phb \in \FC(\L)$, nous noterons $\bar{\phb}$ sa classe 
dans $\FC(\L)/(\RM_{>0})^d$ et nous poserons 
$\Posbar(\bar{\phb})=\Pos(\phb)$. Comme corollaire de la proposition 
\ref{surjection pos}, nous obtenons une classification des ordres 
totaux sur $\L$ (compatibles avec la structure de groupe), 
ce qui est un r\'esultat classique \cite{robbiano}. 
En fait, se donner un ordre 
total sur $\L$ est \'equivalent \`a se donner une partie 
positive $X$ de $\L$ telle que $X \cap (-X)=0$. Notons 
$\FC_0(\L)$ l'ensemble des \'el\'ements $(\ph_1,\dots,\ph_r) \in \FC(\L)$ 
tels que $\L \cap \Ker \ph_r = 0$. 

\bigskip

\begin{coro}\label{ordres totaux}
L'application d\'ecrite dans la proposition \ref{surjection pos} induit 
une bijection entre l'ensemble $\FC_0(\L)/(\RM_{> 0})^d$ et l'ensemble 
des ordres totaux sur $\L$ compatibles avec la structure 
de groupe. 
\end{coro}

\bigskip

\subsection{Fonctorialit\'e} 
Soit $\s : \L' \to \L$ un morphisme de groupes ab\'eliens. 
Le r\'esultat suivant est facile~:

\bigskip

\begin{lem}\label{sigma}
Si $X$ est une partie positive de $\L$, 
alors $\s^{-1}(X)$ est une partie positive de $\L'$.
\end{lem}

\begin{proof}
Soit $X \in \positif(\L)$. Posons $X' = \s^{-1}(X)$. 
Montrons que $X' \in \positif(\L')$.

\medskip

(1) On a $X' \cup (-X') = \s^{-1}(X \cup (-X))=\s^{-1}(\L)=\L'$.

\smallskip

(2) Si $\l'$ et $\mu'$ sont deux \'el\'ements de $X'$, alors 
$\s(\l')$ et $\s(\mu')$ appartiennent \`a $X$. Donc $\s(\l')+\s(\mu') \in X$. 
En d'autres termes, $\l' + \mu' \in X'$. Donc $X'+X' \subseteq X'$. 

\smallskip

(3) On a $X' \cap (-X') = \s^{-1}(X \cap (-X))$, donc l'ensemble 
$X' \cap (-X')$ est un sous-groupe de $\L'$.
\end{proof}

\bigskip

\begin{coro}\label{pos inclusion}
Si $\L'$ est un sous-groupe de $\L$ et si $X$ est une partie 
positive de $\L$, alors $X \cap \L'$ est une partie positive de $\L'$. 
\end{coro}

\bigskip

Si $\s : \L' \to \L$ est un morphisme de groupes ab\'eliens, nous noterons 
$\s^* : \positif(\L) \to \positif(\L')$, $X \mapsto \s^{-1}(X)$ induite 
par le lemme \ref{sigma}. Si $\t : \L'' \to \L'$ 
est un morphisme de groupe ab\'eliens, il est alors facile 
de v\'erifier que 
\equat\label{composition}
(\t \circ \s)^* = \s^* \circ \t^*.
\endequat

D'autre part, si on note $V' = \RM \otimes_\ZM \L'$, alors 
$\s$ induit une application $\RM$-lin\'eaire 
$\s_\RM : V' \longto V$ dont nous noterons 
$\lexp{t}{\s_\RM} : V^* \longto V^{\prime *}$ l'application duale 
et $\lexp{t}{\sigba}_\RM : V^*/\RM_{>0} \longto V^{\prime *}/\RM_{> 0}$ 
l'application (continue) induite. 
Il est alors facile de v\'erifier que le diagramme
\equat\label{diag sigma}
\diagram
V^*/\RM_{>0} \rrto^{\DS{\Posbar}} \ddto^{\DS{\lexp{t}{\sigba}_\RM}} && 
\positif(\L) \rrto^{\DS{\pi}} \ddto^{\DS{\s^*}} && 
V^*/\RM_{>0} \ddto^{\DS{\lexp{t}{\sigba}_\RM}} \\
&&&&\\
V^{\prime *}/\RM_{>0} \rrto^{\DS{\Posbar'}} && 
\positif(\L') \rrto^{\DS{\pi'}} && 
V^{\prime *}/\RM_{>0}
\enddiagram
\endequat
est commutatif (o\`u $\Posbar'$ et $\pi'$ sont les analogues de $\Posbar$ et $\pi$ 
pour le r\'eseau $\L'$).

\bigskip

\section{Topologie sur $\positif(\L)$}

\medskip

Dans cette section, nous allons d\'efinir sur $\positif(\L)$ une topologie 
et en \'etudier les propri\'et\'es. Nous allons notamment montrer que la 
plupart des applications introduites dans la section pr\'ec\'edente 
($\Posbar$, $\pi$, $i_{\phba}$,...) sont continues. 

\bigskip

\subsection{D\'efinition}
Si $E$ est une partie de $\L$, 
nous poserons 
$$\UC(E)=\{X \in \positif(\L)~|~X \cap E=\vide\}.$$
Si $\l_1$,\dots, $\l_n$ sont des \'el\'ements de $\L$, nous 
noterons pour simplifier $\UC(\l_1,\dots,\l_n)$ l'ensemble 
$\UC(\{\l_1,\dots,\l_n\})$. Si cela est n\'ecessaire, nous noterons ces ensembles 
$\UC_\L(E)$ ou $\UC_\L(\l_1,\dots,\l_n)$. On a alors
\equat\label{u inter}
\UC(E)=\bigcap_{\l \in E} \UC(\l).
\endequat
Notons que 
\equat\label{u vide}
\UC(\vide)=\positif(\L)\qquad\text{et}\qquad \UC(\L)=\{\vide\}.
\endequat
D'autre part, si $(E_i)_{i \in I}$ est une famille de parties de $\L$, alors 
\equat\label{u intersection}
\bigcap_{i \in I} \UC(E_i) = \UC\bigl(\bigcup_{i \in I} E_i\bigr).
\endequat
Une partie $\UC$ de $\positif(\L)$ sera dite {\it ouverte} si, 
pour tout $X \in \UC$, il existe une partie {\bf finie} $E$ de $\L$ telle 
que $X \in \UC(E)$ et $\UC(E) \subset \UC$. L'\'egalit\'e 
\ref{u intersection} montre que cela d\'efinit bien une topologie 
sur $\positif(\L)$. 

\bigskip

\begin{prop}\label{connexite}
Si $\UC$ est un ouvert de $\positif(\L)$ contenant $\L$, alors 
$\UC=\positif(\L)$. En particulier, $\positif(\L)$ est connexe. 
Si $\L \neq 0$, alors il n'est pas s\'epar\'e. 
\end{prop}

\bigskip

\begin{proof}
Par d\'efinition, il existe une partie finie $E$ de $\L \setminus \L$ 
telle que $\UC(E) \subseteq \UC$. Mais on a forc\'ement $E=\vide$, 
donc $\UC=\positif(\L)$ d'apr\`es \ref{u vide}. D'o\`u le 
r\'esultat. 

Le fait que $\positif(\L)$ 
n'est pas s\'epar\'e (lorsque $\L \neq 0$) en d\'ecoule~: 
le point $\L$ de $\positif(\L)$ ne peut \^etre 
s\'epar\'e d'aucun autre.
\end{proof}

\bigskip

\example{Z} 
L'espace topologique $\positif(\ZM)$ n'a que trois points~: 
$\ZM$, $\ZM_{\ge 0}$ et $\ZM_{\le 0}$. Sur ces trois points, 
seul $\ZM$ est un point ferm\'e et $\ZM_{\ge 0}$ et $\ZM_{\le 0}$ 
sont des points ouverts (en effet, $\{\ZM_{\ge 0}\} = \UC(-1)$ 
et $\{\ZM_{\le 0}\} = \UC(1)$).\finl

\bigskip

Il est clair que la topologie sur $\positif(\L)$ d\'efinie ci-dessus 
est la topologie induite par une topologie sur l'ensemble des parties 
de $\L$ (d\'efinie de fa\c{c}on analogue)~: cette derni\`ere est tr\`es 
grossi\`ere mais sa restriction \`a $\positif(\L)$ est plus int\'eressante. 

\medskip

Nous aurons besoin de la propri\'et\'e suivante des ensembles $\UC(E)$~:

\bigskip

\begin{lem}\label{uce vide}
Soit $E$ une partie {\bfit finie} de $\L$. Alors les assertions suivantes 
sont \'equivalentes~:
\begin{itemize}
\itemth{1} $\UC(E)=\vide$.

\itemth{2} Il existe $n \ge 1$, $\l_1$,\dots, $\l_n \in E$ et 
$r_1$,\dots, $r_n \in \ZM_{>0}$ tels que $\DS{\sum_{i=1}^n r_i\l_i=0}$. 
%$r_1\l_1 + \cdots + r_n \l_n = 0$.

\itemth{3} Il n'existe pas de forme lin\'eaire $\ph$ sur $V^*$ 
telle que $\ph(E) \subset \RM_{>0}$.
\end{itemize}
\end{lem}

\bigskip

\begin{proof}
S'il existe une forme lin\'eaire $\ph$ sur $V^*$ 
telle que $\ph(E) \subseteq \RM_{>0}$, alors $\Pos(-\ph) \in \UC(E)$, et 
donc $\UC(E) \neq \vide$. Donc (1) $\imp$ (3). 

\medskip

Supposons trouv\'es $\l_1$,\dots, $\l_n \in E$ et $r_1$,\dots, $r_n \in \ZM_{>0}$ 
tels que $r_1\l_1 + \cdots + r_n \l_n = 0$. Alors, si $X \in \UC(E)$, 
on a $-\l_2$,\dots, $-\l_n \in X$. Mais 
$r_1\l_1 = -r_2\l_2 - \cdots - r_n \l_n \in X$, 
donc $\l_1 \in X$, ce qui contredit l'hypoth\`ese. Donc (2) $\imp$ (1).

\medskip

Il nous reste \`a montrer que (3) $\imp$ (2). Supposons que (2) n'est pas vraie. 
Nous allons montrer qu'alors (3) n'est pas vraie en raisonnant par r\'ecurrence sur
la dimension de $V$ (c'est-\`a-dire le rang de $\L$). 
Posons 
$$\CC=\{t_1\l_1+\cdots t_n \l_n~|~n \ge 1,
~\l_1,\dots, \l_n \in \L,~t_1,\dots,t_n \in \RM_{>0}\}.$$
Alors $\CC$ est une partie convexe de $V$ contenant $E$ et, d'apr\`es 
le lemme \ref{solutions} et le fait que (2) ne soit pas vraie, on a 
$0 \not\in \CC$. Par cons\'equent, il r\'esulte du th\'eor\`eme de Hahn-Banach 
qu'il existe une forme lin\'eaire non nulle $\ph$ telle que 
$\ph(\CC)\subseteq \RM_{\ge 0}$. Posons $\L'=(\Ker \ph)\cap \L$ et 
$E' = E \cap \L$. Alors l'assertion (2) pour $E'$ n'est pas vraie elle aussi donc, 
par hypoth\`ese de r\'ecurrence, il existe un forme lin\'eaire $\psi$ sur 
$V'=\RM \otimes_\ZM \L' \subseteq V$ telle que $\psi(E') \subset \RM_{>0}$. 
Soit $\psit$ une extension de $\psi$ \`a $V$. 
Puisque $\ph(E \setminus E') \subset \RM_{>0}$, il existe $\e > 0$ tel que 
$\ph(\l) + \e \psit(\l) > 0$ pour tout $\l \in E\setminus E'$. 
Mais on a aussi, si $\l \in E'$, $\ph(\l)+\e\psit(\l) = \e \psit(\l) > 0$. 
Donc $(\ph + \e \psit)(E) \subset \RM_{>0}$.
\end{proof}

\bigskip

Si $\EC$ est une partie de $\positif(\L)$, nous noterons 
$\overline{\EC}$ son adh\'erence dans $\positif(\L)$.

\bigskip

\begin{coro}\label{adherence u}
Soit $E$ une partie {\bfit finie} de $\L$ telle que $\UC(E) \neq \vide$. Alors
$$\overline{\UC(E)} = \positif(\L) \setminus \Bigl(\bigcup_{\l \in E} \UC(-\l)\Bigr).$$
\end{coro}

\bigskip

\begin{proof}
Posons 
$$\OC=\bigcup_{\l \in E} \UC(-\l)$$
$$\FC=\positif(\L) \setminus \OC. \leqno{\text{et}}$$
Alors $\FC$ est ferm\'e dans $\positif(\L)$ et contient $\UC(E)$. 
Donc $\overline{\UC(E)} \subseteq \FC$. 

R\'eciproquement, soit $X \in \positif(\L) \setminus \overline{\UC(E)}$. 
Nous devons montrer que 
$$X \in \OC.\leqno{(?)}$$
Puisque $\positif(\L) \setminus \overline{\UC(E)}$ est un ouvert, il existe 
une partie finie $F$ de $\L$ telle que $X \in \UC(F)$ et 
$\UC(F) \subseteq \OC$. En particulier, $\UC(F) \cap \UC(E) = \vide$. 
En d'autres termes, d'apr\`es \ref{u intersection}, on a 
$\UC(E \cup F) = \vide$. 
Donc, d'apr\`es le lemme \ref{uce vide}, il existe $m \ge 0$, $n \ge 0$, 
$\l_1$,\dots, $\l_m \in E$, $\mu_1$,\dots, $\mu_n \in F$, 
$r_1$,\dots, $r_m$, $s_1$,\dots, $s_n \in \ZM_{>0}$ tels que
$$r_1\l_1 + \cdots + r_m \l_m + s_1 \mu_1 + \cdots + s_n \mu_n=0$$
et $m+n \ge 1$. En fait, comme $\UC(E)$ et $\UC(F)$ sont toutes deux non vides, 
il d\'ecoule du lemme \ref{uce vide} que $m$, $n \ge 1$. 

Si $X \not\in \OC$, alors $-\l_i \in X$ pour tout $i$, ce qui implique 
que $s_1 \mu_1 + \cdots + s_n \mu_n \in X$. Cela ne peut se produire que si au moins 
l'un des $\mu_j$ appartient \`a $X$, mais c'est impossible car $F \cap X =\vide$. 
D'o\`u (?).
\end{proof}

\bigskip

\example{viditude} 
Le corollaire \ref{adherence u} n'est pas forc\'ement vrai si $\UC(E)=\vide$. 
En effet, si $\l \in \L \setminus\{0\}$, alors $\UC(\l,-\l)=\vide$ mais, 
du moins lorsque $\dim V \ge 2$, on a $\UC(\l) \cup \UC(-\l) \neq \positif(\L)$.\finl

\bigskip

\subsection{Fonctorialit\'e} 
Si $\L'$ est un autre r\'eseau et si $\s : \L' \to \L$ est un 
morphisme de groupes et si $E'$ est une partie de $\L'$, alors 
\equat\label{sigma continue}
(\s^*)^{-1}\bigl(\UC_{\L'}(E')\bigr)=\UC_\L\bigl(\s(E')\bigr).
\endequat
%\begin{quotation}
\begin{proof}[Preuve de \ref{sigma continue}] 
Soit $X$ une partie positive de $\L$. Alors 
$X \in (\s^*)^{-1}\bigl(\UC_{\L'}(E')\bigr)$ 
(resp. $X \in \UC_\L\bigl(\s(E')\bigr)$) 
si et seulement si $\s^{-1}(X) \cap E' = \vide$ 
(resp. $X \cap \s(E') = \vide$). 
Il est alors facile de v\'erifier que ces deux derni\`eres conditions sont 
\'equivalentes.
\end{proof}
%\end{quotation}

Cela implique le r\'esultat suivant~:

\bigskip

\begin{prop}\label{continu}
L'application $\s^* : \positif(\L) \to \positif(\L')$ est continue. 
\end{prop}

\bigskip

%\noindent{\sc Exemple - } 
%Soit $\s : \ZM^2 \to \ZM$, $(m,n) \mapsto m$. Alors $\s$ est un morphisme 
%de groupes ab\'eliens dont le noyau n'est pas fini. Nous allons montrer 
%ici que $\s^*$ n'est pas ouverte. Notons que 
%$$\s^*(\UC_\ZM(-1)) = \UC_{\ZM^2}(\{-1\} \times \ZM)$$
%d'apr\`es \ref{sigma inverse}. On en d\'eduit facilement que 
%$$\s^*(\UC_\ZM(-1))=\{X,Y_+,Y_-\},$$ 
%o\`u $X=\ZM_{\ge 0} \times \ZM$, 
%$Y_+ = (\ZM_{> 0} \times \ZM) \cup (\{0\} \times \ZM_{\ge 0})$ et 
%$Y_- = (\ZM_{> 0} \times \ZM) \cup (\{0\} \times \ZM_{\le 0})$. 
%Il est alors tout aussi facile de v\'erifier que 
%$\s^*(\UC_\ZM(-1))$ n'est pas ouvert.\finl
%
%\bigskip

\subsection{Continuit\'e} 
D'apr\`es la section \ref{section positive}, nous avons \'equip\'e 
l'espace topologique $\positif(\L)$ de 
deux applications $\Posbar : V^*/\RM_{> 0} \to \positif(\L)$ et 
$\pi : \positif(\L) \to V^*/\RM_{>0}$ telles que 
$\pi \circ \Posbar = \Id_{V^*/\RM_{>0}}$. Nous montrerons dans le 
th\'eor\`eme \ref{pi pos continues} que ces applications 
sont continues (lorsque $V^*/\RM_{>0}$ est bien s\^ur muni de 
la topologie quotient) et nous en d\'eduirons quelques autres 
propri\'et\'es topologiques de ces applications. 
Avant cela, introduisons la notation suivante~: si $E$ 
est une partie finie de $\L$, on pose 
$$\VC(E)=\{\phba\in V^*/\RM_{> 0}~|~\forall~\l \in E,~\ph(\l) < 0\}.$$
Si cela est n\'ecessaire, nous noterons $\VC_\L(E)$ l'ensemble $\VC(E)$. 
Alors
$$p^{-1}(\VC(E))=\{\ph \in V^*~|~\forall~\l \in E,~\ph(\l) < 0\}.$$
Donc $p^{-1}(\VC(E))$ est ouvert, donc $\VC(E)$ est ouvert 
dans $V^*/\RM_{>0}$ par d\'efinition de la topologie quotient. 

\bigskip

\begin{prop}\label{pi pos continues}
Les applications $\Posbar$ et $\pi$ sont continues. De plus~:
\begin{itemize}
\itemth{a} $\Posbar$ induit un hom\'eomorphisme sur son image.

\itemth{b} L'image de $\Posbar$ est dense dans $\positif(\L)$.
%
%\itemth{c} $\positif(\L)$ est quasi-compact. 
\end{itemize}
\end{prop}

\begin{proof} 
Soit $E$ une partie finie de $\L$. Alors 
\equat\label{pos uc}
\Posbar^{-1}(\UC(E))=\VC(E).
\endequat
Donc $\Posbar^{-1}(\UC(E))$ est un ouvert de $V^*/\RM_{>0}$. Donc 
$\Posbar$ est continue. 

Montrons maintenant que $\pi$ est continue. Nous proc\`ederons 
par \'etapes~:

\medskip

\begin{quotation}
\begin{lem}\label{base ouverts}
Les $\VC(E)$, o\`u $E$ parcourt l'ensemble des parties 
finies de $\L$, forment une base d'ouverts de $V^*/\RM_{>0}$. 
\end{lem}

\begin{proof}[Preuve du lemme \ref{base ouverts}]
Soit $\UC$ un ouvert de $V^*/\RM_{>0}$ et soit $\ph$ une forme lin\'eaire 
sur $V$ telle que $\phba \in \UC$. 
Nous devons montrer qu'il existe une partie finie $E$ de $V$ 
telle que $\phba \in \VC(E)$ et $\VC(E) \subset \UC$. 

Si $\ph=0$, alors $\UC=V^*/\RM_{>0}$ et le r\'esultat est clair. 
Nous supposerons donc que $\ph \neq 0$. Il existe alors $\l_0 \in \L$ 
tel que $\ph(\l_0) > 0$. Quitte \`a remplacer $\ph$ par un multiple 
positif, on peut supposer que $\ph(\l_0)=1$. Notons $\HC_0$ l'hyperplan 
affine $\{\psi \in V^*~|~\psi(\l_0)=1\}$. Alors l'application naturelle 
$\HC_0 \to V^*/\RM_{>0}$ induit un hom\'eomorphisme 
$\nu : \HC_0 \stackrel{\sim}{\longrightarrow} \VC(-\l_0)$. 
De plus, $\ph=\nu^{-1}(\ph) \in \HC_0$. Donc $\ph \in \nu^{-1}(\UC \cap \VC(-\l_0))$. 
Il suffit dons de v\'erifier que les intersections finies de demi-espaces 
ouverts {\it rationnels} (i.e. de la forme $\{\psi \in \HC_0~|~\psi(\l) > n\}$ o\`u 
$n \in \ZM$ et $\l \in \L \setminus \ZM \l_0$) 
forment une base de voisinages de l'espace affine $\HC_0$, ce qui est imm\'ediat.
\end{proof}
\end{quotation}

\medskip

Compte tenu du lemme \ref{base ouverts}, il suffit de montrer que, 
si $E$ est une partie finie de $\L$, alors $\pi^{-1}(\VC(E))$ 
est un ouvert de $\positif(\L)$. De plus, 
$$\VC(E)=\bigcap_{\l \in E} \VC(\l).$$
Par cons\'equent, la continuit\'e de $\pi$ 
d\'ecoulera du lemme suivant:

\medskip

\begin{quotation}
\begin{lem}\label{image inverse}
Si $\l \in \L$, alors $\pi^{-1}(\VC(\l))$ est un ouvert de 
$\positif(\L)$.
\end{lem}

\begin{proof}[Preuve du lemme \ref{image inverse}]
Soit $X \in \pi^{-1}(\VC(\l))$ et soit $\ph=\pi(X)$. 
Par d\'efinition, $\ph(\l) < 0$ et donc $\l \not\in X$. 
Soit $e_1$,\dots, $e_n$ une $\ZM$-base de $\L$. Il existe un entier 
naturel non nul $N$ tel que $\ph(\l \pm \DS{\frac{1}{N}e_i}) < 0$ pour 
tout $i$. Quitte \`a remplacer $\l$ par $N\l$, on peut supposer que 
$\ph(\l \pm e_i) < 0$ pour tout $i$. On pose alors 
$$E=\{\l+e_1,\l-e_1,\dots,\l+e_n,\l-e_n\}.$$
Alors $X \in \UC(E)$ par construction. Il reste \`a montrer 
que $\UC(E) \subseteq \pi^{-1}(\VC(\l))$. Soit $Y \in \UC(E)$ 
et posons $\psi = \pi(Y)$. Supposons de plus que $\psi \not\in \VC(\l)$. 
On a alors $\psi(\l) \ge 0$. D'autre part, $\psi(\l \pm e_i) \le 0$ pour tout $i$. 
Cela montre que $2\psi(\l)=\psi(\l+e_1)+\psi(\l-e_1) \le 0$, et donc 
$\psi(\l)=\psi(\l+e_i)=0$, et donc $\psi(e_i)=0$ pour tout $i$. Donc $\psi$ 
est nulle et donc $Y=\L$, ce qui contredit le fait que 
$Y \in \UC(E)$. Cela montre donc que $\psi \in \VC(\l)$, comme attendu.
\end{proof}
\end{quotation}

\medskip

Puisque $\pi$ et $\Posbar$ sont continues et v\'erifient 
$\pi \circ \Posbar = \Id_{V^*/\RM_{>0}}$, $\Posbar$ induit un 
hom\'eomorphisme sur son image. D'o\`u (a).

\medskip

L'assertion (b) d\'ecoule du lemme suivant (qui est une cons\'equence imm\'ediate 
de l'\'equivalence entre (1) et (3) dans le lemme \ref{uce vide}) et de \ref{pos uc}~:
\begin{quotation}
\begin{lem}\label{non vide}
Soit $E$ une partie finie de $\L$ telle que $\UC(E) \neq \vide$. 
Alors $\VC(E) \neq \vide$. 
\end{lem}

%\begin{proof}[Preuve du lemme \ref{non vide}]
%Nous pouvons supposer que $E$ est non vide et nous raisonnons par 
%r\'ecurrence sur la dimension de $V$ (c'est-\`a-dire le rang de $\L$), 
%le cas des dimensions $0$ et $1$ \'etant triviaux. %
%
%Soit donc $X \in \UC_\L(E)$ et soit $\ph=\pi(X)$. Alors $\ph(\l) \le 0$ 
%pour tout $\l \in E$. Si $\ph \in \VC_\L(E)$, 
%il n'y a plus rien \`a montrer. Supposons donc que $\ph\not\in \VC_\L(E)$. 
%Alors il existe $\l \in E$ tel que $\ph(\l)=0$. Posons $\L'=\Ker \ph|_\L$, 
%$X'=X \cap \L'$ et $E'=E \cap \L'$. Alors $X' \in \UC_{\L'}(E')$, ce 
%qui montre que $\UC_{\L'}(E') \neq \vide$. L'hypoth\`ese de r\'ecurrence 
%nous assure l'existence d'une forme lin\'eaire $\psi$ sur $V'=\RM \otimes_\ZM \L'$ 
%telle que $\psi(\l') < 0$ pour tout $\l' \in E'$. Choisissons un extension 
%$\psit$ de $\psi$ \`a $V$. Alors il existe $\e > 0$ tel que 
%$\ph(\l)+\e \psit(\l) < 0$ pour tout $\l \in E\setminus E'$. Mais, si $\l' \in E'$, 
%alors $\ph(\l')+\e \psit(\l')=\e\psi(\l') < 0$, donc finalement, 
%$\ph + \e \psit \in \VC_\L(E)$. 
%\end{proof}
\end{quotation}

La preuve de la proposition \ref{pi pos continues} est termin\'ee. 
\end{proof}

\bigskip

Nous allons maintenant \'etudier les propri\'et\'es topologiques des 
applications $i_\phba$. 

\bigskip

\begin{prop}\label{proprietes topologie}
Soit $\ph \in V^*$ et supposons $\ph \neq 0$. Alors~:
\begin{itemize}
\itemth{a} 
$\DS{\bigcap_{\stackrel{\UC {\mathrm{~ouvert~de~}}\positif(\L)}{\Pos(\ph) \in \UC}} 
\UC = i_{\phba}\bigl(\positif(\Ker \ph|_\L)\bigr)} = \pi^{-1}(\phba)$.

\itemth{d} $i_{\phba}$ est continue et induit un hom\'eomorphisme sur son image.
\end{itemize}
\end{prop}

\begin{proof}
(a) Notons $I_\ph$ l'image de $i_\phba$. On a alors, d'apr\`es 
le lemme \ref{equivalent separant}, 
$$I_\ph=\{X \in \positif(\L)~|~\Pos^+(\ph) \subseteq X\}.\leqno{(*)}$$
Si $\UC$ est un ouvert contenant $\Pos(\ph)$, 
alors il existe une partie finie $E$ de $\L \setminus \Pos(\ph)$ 
telle que $\UC(E) \subseteq \UC$. Mais, si $X$ est dans l'image de 
$i_\phba$, alors $X \subseteq \Pos(\ph)$, donc $X \cap E = \vide$. 
Et donc $X \in \UC$, ce qui montre que 
$$I_\ph \subseteq
\bigcap_{\stackrel{\UC {\mathrm{~ouvert~de~}}\positif(\L)}{\Pos(\ph) \in \UC}} 
\UC.$$
Montrons l'inclusion r\'eciproque. Soit $X \in \positif(\L)$ tel 
que $X \not\in I_\ph$. Posons $\psi=\pi(X)$. Alors $\psiba \neq \phba$ 
donc, d'apr\`es la preuve du th\'eor\`eme \ref{hahn banach}, il existe 
$\l \in \L$ tel que $\ph(\l) < 0$ et $\psi(\l) > 0$. On a donc, 
d'apr\`es le lemme \ref{equivalent separant}, $X \not\in \UC(\l)$. 
D'autre part, $\Pos(\ph) \in \UC(\l)$. D'o\`u (a).

\medskip

Montrons (b). 
On note $\pi_\phba : I_\ph \to \positif(\Ker \ph|_\L)$, 
$X \mapsto X \cap \Ker \ph|_\L$. D'apr\`es la 
proposition \ref{fibres pi}, $\pi_\phba$ est la bijection 
r\'eciproque de $i_\phba : \positif(\Ker \ph|_\L) \to I_\ph$. 
Il nous faut donc montrer que $i_\phba$ et $\pi_\phba$ 
sont continues. Si $F$ est une partie finie de 
$\Ker \ph|_\L$, nous noterons $\UC_\phba(F)$ l'analogue de 
l'ensemble $\UC(F)$ d\'efini \`a l'int\'erieur de $\positif(\Ker \ph|_\L)$. 

Soit $E$ une partie finie de $\L$. Nous voulons montrer 
que $i_\phba^{-1}(\UC(E))$ est un ouvert de $\positif(\Ker \ph|_\L)$. 
S'il existe $\l \in E$ tel que 
$\ph(\l) > 0$, alors $\UC(E) \cap I_\ph = \vide$ (voir $(*)$). On peut 
donc supposer que $\ph(\l) \le 0$ pour tout $\l \in E$. 
Il est alors facile de v\'erifier que 
$$i_\phba^{-1}(\UC(E)) = \UC_\phba(E \cap \Ker \ph|_\L).$$
Donc $i_\phba$ est continue. 

Soit $F$ une partie finie de $\Ker \ph|_\L$. Alors 
$$\pi_\phba^{-1}(\UC_\phba(F)) =\UC(F) \cap I_\ph.$$
Donc $\pi_\phba$ est continue. 
\end{proof}

Nous allons r\'esumer dans le th\'eor\`eme suivant 
la plupart des r\'esutats obtenus dans cette sous-section.

\bigskip

\begin{theo}\label{theo:topologie}
Supposons $\L \neq 0$ et soit $\ph \in V^*$, $\ph \neq 0$. 
\begin{itemize}
\itemth{a} $\positif(\L)$ est connexe. Il n'est pas s\'epar\'e si $\L \neq 0$.

\itemth{b} Les applications 
$\pi : \positif(\L) \to V^*/\RM_{>0}$ et $\Posbar : V^*/\RM_{>0} \to \positif(\L)$ 
sont continues et v\'erifient $\pi \circ \Posbar = \Id_{V^*/\RM_{>0}}$.

\itemth{c} $\Posbar$ induit un hom\'eomorphisme sur son image~; cette 
image est dense dans $\positif(\L)$.

\itemth{d} $\pi^{-1}(\phba)$ est l'intersection des voisinages de $\Pos(\ph)$ 
dans $\positif(\L)$.

\itemth{e} $i_\phba$ est un hom\'eomorphisme.
\end{itemize}
\end{theo}

\bigskip

\section{Arrangements d'hyperplans}

\medskip

L'application continue $\Pos : V^* \to \positif(\L)$ 
a une image dense. Nous allons \'etudier ici comment se 
transpose la notion d'arrangement d'hyperplans (et les 
objets attach\'es~: facettes, chambres, support...) 
\`a l'espace topologique $\positif(\L)$ \`a travers $\Pos$. 
Cela nous permettra d'\'enoncer les conjectures sur 
les cellules de Kazhdan-Lusztig sous la forme la plus g\'en\'erale 
possible \cite[Conjectures A et B]{bonnafe}.

\bigskip

\subsection{Sous-espaces rationnels} 
Si $E$ est une partie de $\L$, on pose 
$$\LC(E)=\{X \in \positif(\L)~|~E \subset X \cap (-X)\}.$$
Si cela est n\'ecessaire, nous le noterons $\LC_\L(E)$. 
On appelle {\it sous-espace rationnel} de $\positif(\L)$ 
toute partie de $\positif(\L)$ de la forme $\LC(E)$, o\`u 
$E$ est une partie de $\L$. Si $\l \in \L\setminus \{0\}$, 
on notera $\HC_\l$ le sous-espace rationnel $\LC(\{\l\})$~: 
un tel sous-espace rationnel sera appel\'e un {\it hyperplan 
rationnel}. Notons que 
\equat\label{partage espace}
\positif(\L)= \UC(\l) \dotcup \HC_\l \dotcup \UC(-\l).
\endequat
La proposition suivante justifie quelque peu la terminologie~:

\begin{prop}\label{topo espace}
Soit $E$ est une partie de $\L$. Notons $\L(E)$ le sous-r\'eseau 
$\L \cap \sum_{\l \in E} \QM E$ de $\L$ et soit $\s_E : \L \to \L/\L(E)$ 
l'application canonique. Alors~:
\begin{itemize}
\itemth{a} $\LC(E)=\DS{\bigcap_{\l \in E\setminus\{0\}} \HC_\l} = 
\{X \in \positif(\L)~|~\L(E) \subseteq X\}$.

\itemth{b} $\LC(E)$ est ferm\'e dans $\positif(\L)$. 

\itemth{c} $\Pos^{-1}(\LC(E))=\{\ph \in V^*~|~\forall~\l \in E,~\ph(\l)=0\}=E^\perp$.

\itemth{d} $\Posbar(\pi(\LC(E)) \subseteq \LC(E)$. 

\itemth{e} $\Posbar^{-1}(\LC(E))=\pi(\LC(E))$.

\itemth{f} $\LC(E)=\overline{\Pos(\Pos^{-1}(\LC(E)))}$ .

\itemth{g} L'application $\s_E^* : \positif(\L/\L(E)) \to \positif(\L)$ 
a pour image $\LC(E)$ et induit un hom\'eomorphisme 
$\positif(\L/\L(E)) \stackrel{\sim}{\longrightarrow} \LC(E)$. 
\end{itemize}
\end{prop}

\begin{proof}
La premi\`ere \'egalit\'e de (a) est imm\'ediate. La deuxi\`eme 
d\'ecoule de la proposition \ref{proprietes positives} (c). 
(b) d\'ecoule de (a) et de \ref{partage espace}.
(c) est tout aussi clair. 

\medskip

(d) Si $X \in \Posbar(\pi(\LC(E))$, alors il existe 
$Y \in \LC(E)$ tel que $X=\Posbar(\pi(Y))$. Posons $\phba = \pi(Y)$, 
o\`u $\ph \in V^*$. Alors $E \subseteq Y \cap (-Y)$ et 
$Y \subseteq \Pos(\ph)$. Or, $\ph(\l)=0$ si $\l \in Y \cap (-Y)$, 
donc $X=\Pos(\ph) \in \LC(E)$. 

\medskip

(e) D'apr\`es (d), on a $\pi(\L(E)) \subseteq \Posbar^{-1}(\LC(E))$. 
R\'eciproquement, soit $\ph$ un \'el\'ement de $\Pos^{-1}(\LC(E))$, alors 
$\phba=\pi(\Posbar(\phba)) \in \pi(\LC(E))$. D'o\`u (e).

\medskip

(f) Notons $\FC_E=\Pos(\Pos^{-1}(\LC(E)))$. 
On a $\FC_E \subseteq \LC(E)$ donc 
il d\'ecoule du (a) que $\overline{\FC}_E \subseteq \LC(E)$. 
R\'eciproquement, soit $F$ une partie finie de $\L$ telle que $\UC(F) \cap \FC_E = \vide$. 
Nous devons montrer que $\UC(F) \cap \LC(E)=\vide$. Or, le fait que $\UC(F) \cap \FC_E=\vide$ 
est \'equivalent \`a l'assertion suivante (voir (c))~:
$$\forall~\ph \in E^\perp,~\forall~\l \in F,~\ph(\l) \ge 0.$$
Or, si $\ph \in E^\perp$, alors $-\ph \in E^\perp$, ce qui implique que~:
$$\forall~\ph \in E^\perp,~\forall~\l \in F,~\ph(\l) = 0.$$
En d'autres termes, $F \subseteq (E^\perp)^\perp \cap \L = \L(E)$. Mais, 
si $X \in \LC(E)$, alors $\L(E) \subseteq X$ d'apr\`es (a). Donc 
$X \not\in \UC(F)$, comme esp\'er\'e.

\medskip

(g) Le fait que l'image de $\s_E^*$ soit 
$\LC(E)$ d\'ecoule de (a). D'autre part, $\s_E^*$ est continue d'apr\`es la 
proposition \ref{continu}. Notons 
$$\fonction{\g_E}{\LC(E)}{\positif(\L/\L(E))}{X}{X/\L(E).}$$
Alors $\g_E$ est la r\'eciproque de $\s_E^*$. Il ne nous reste 
qu'\`a montrer que $\g_E$ est continue. Soit donc $F$ une partie finie de $\L/\L(E)$ 
et notons $\Fti$ un ensemble de repr\'esentants des \'el\'ements de $F$ dans $\L$.
On a 
\eqna
\g_E^{-1}(\UC_{\L/\L(E)}(F))&=&\{X \in \LC(E)~|~\forall~\l \in F,~\l \not\in X/\L(E)\}\\
&=&\{X \in \positif(\L)~|~\L(E) \subseteq X\text{ et }\forall~\l \in F,~\l \not\in X/\L(E)\}\\
&=&\{X \in \positif(\L)~|~\L(E) \subseteq X\text{ et }\forall~\l \in \Fti,~\l \not\in X\}\\
&=&\{X \in \LC(E)~|~\forall~\l \in \Fti,~\l \not\in X\}\\
&=&\LC(E) \cap \UC_\L(\Fti).
\endeqna
Donc $\g_E^{-1}(\UC_{\L/\L(E)})$ est un ouvert de $\LC(E)$. Cela montre la continuit\'e de 
$\g_E$.
\end{proof}

\bigskip

\subsection{Demi-espaces} 
Soit $\HC$ un hyperplan rationnel de $\positif(\L)$ et soit $\l\in \L\setminus\{0\}$ 
tel que $\HC=\HC_\l$. D'apr\`es \ref{partage espace}, 
l'hyperplan $\HC$ nous d\'efinit une unique relation d'\'equivalence 
$\smile_{\!\HC}$ sur $\positif(\L)$ pour laquelle les classes d'\'equivalence 
sont $\UC(\l)$, $\HC$ et $\UC(-\l)$~: notons que cette relation 
ne d\'epend pas du choix de $\l$. De plus~:

\bigskip

\begin{prop}\label{composantes connexes}
$\HC$ est un ferm\'e de $\positif(\L)$ et $\UC(\l)$ et $\UC(-\l)$ 
sont les composantes connexes de $\positif(\L)\setminus \HC$. De plus
$$\overline{\UC(\l)}=\UC(\l) \cup \HC_\l.$$
\end{prop}

\bigskip

\begin{proof}
La derni\`ere assertion est un cas particulier du corollaire 
\ref{adherence u}. 

\medskip

Montrons pour finir que $\UC(\l)$ est connexe. Soient $\UC$ et $\VC$ deux ouverts 
de $\UC(\l)$ tels que $\UC(\l)=\UC \coprod \VC$. Alors 
$$\Pos^{-1}(\UC(\l))=\Pos^{-1}(\UC) \coprod \Pos^{-1}(\VC).$$
Mais $\Pos^{-1}(\UC(\l))=\{\ph \in V^*~|~\ph(\l) < 0\}$. Donc 
$\Pos^{-1}(\UC(\l))$ est connexe. Puisque $\Pos$ est continue, 
cela implique que $\Pos^{-1}(\UC)=\vide$ ou $\Pos^{-1}(\VC)=\vide$. 
Le lemme \ref{non vide} implique que $\UC=\vide$ ou $\VC=\vide$.
\end{proof}

\bigskip

Si $X \in \positif(\L)$, nous noterons $\DC_\HC(X)$ la classe 
d'\'equivalence de $X$ sous la relation $\smile_{\! \HC}$. Il r\'esulte 
de la proposition \ref{composantes connexes} que $\overline{\DC_\HC(X)}$ est 
une r\'eunion de classes d'\'equivalences pour $\smile_{\! \HC}$.

\bigskip

\subsection{Arrangements} 
Nous travaillerons d\'esormais sous l'hypoth\`ese suivante~:

\bigskip

\begin{quotation}
{\it Fixons maintenant, et ce jusqu'\`a la fin de cette section, 
un ensemble {\bfit fini} $\AG$ d'hyperplans rationnels de $\positif(\L)$.}
\end{quotation}

\bigskip

Nous allons red\'efinir, dans notre espace $\positif(\L)$, les notions 
de {\it facettes}, {\it chambres} et {\it faces} associ\'ees \`a $\AG$, 
de fa\c{c}on similaire \`a ce qui se fait pour les arrangements 
d'hyperplans dans un espace r\'eel 
\cite[Chapitre V, \S 1]{bourbaki}. Les propri\'et\'es des 
applications $\pi$ et $\Posbar$ \'etablies pr\'ec\'edemment 
permettent facilement de d\'emontrer les r\'esultats 
analogues. 

Nous d\'efinissons la relation $\smile_\AG$ sur $\positif(\L)$ de la 
fa\c{c}on suivante~:
si $X$ et $Y$ sont deux \'el\'ements de $\positif(\L)$, nous \'ecrirons 
$X \smile_\AG Y$ si $X \smile_{\!\HC} Y$ pour tout $\HC \in \AG$. 
Nous appellerons {\it facettes} (ou {\it $\AG$-facettes}) les classes 
d'\'equivalence pour la relation $\smile_\AG$. Nous appellerons {\it chambres} (ou
{\it $\AG$-chambres}) les facettes qui ne rencontrent aucun hyperplan de
$\AG$. Si $\FC$ est une facette, nous noterons
$$\LC_{\AG} ( \FC) = \bigcap_{\overset{\HC \in \AG}{\FC \subset \HC}} \HC,$$
avec la convention habituelle que $\LC_{\AG} ( \FC) = \positif (\L)$ si
$\FC$ est une chambre. Nous l'appellerons le {\it support} de $\FC$ et
nous appellerons {\it dimension} de $\FC$ l'entier 
$$\dim \FC = \dim_{\RM} \Pos^{- 1} ( \LC_{\AG} ( \FC)) .$$
De m\^eme nous appellerons {\it codimension} de $\FC$ l'entier
$$\codim \FC=\dim_\RM V - \dim \FC.$$
Avec ces d\'efinitions, une chambre est une facette de codimension $0$.

\bigskip

\begin{prop}\label{facettes}
Soit $\FC$ une facette et soit $X \in \FC$. Alors~:
\begin{itemize}
\itemth{a} $\FC=\DS{\bigcap_{\HC \in \AG}} \DC_\HC(X)$.

\itemth{b} $\overline{\FC}=\DS{\bigcap_{\HC \in \AG}} \overline{\DC_\HC(X)}$.

\itemth{c} $\overline{\FC}$ est la r\'eunion de $\FC$ et de facettes de dimension 
strictement inf\'erieures.

\itemth{d} Si $\FC'$ est une facette telle que $\overline{\FC}=\overline{\FC}'$, 
alors $\FC=\FC'$.
\end{itemize}
\end{prop}

\bigskip

\begin{proof}
(a) est une cons\'equence des d\'efinitions. Montrons (b). Posons 
$$\AG_1=\{\HC \in \AG~|~\FC \subseteq \HC\}$$
$$\AG_2=\AG \setminus \AG_1.$$
Pour tout $\HC \in \AG$, on fixe un \'el\'ement $\l(\HC) \in \L$ tel que 
$\HC=\HC_{\l(\HC)}$~: si de plus $\HC \in \AG_2$, on choisit $\l(\HC)$ de sorte 
que $\FC \subseteq \UC(\l(\HC))$. On pose 
$$E_i=\{\l(\HC)~|~\HC \in \AG_i\}.$$ 
Par cons\'equent, 
$$\FC=\LC \cap \UC(E_i).$$
Puisque $\LC$ est ferm\'e, $\overline{\FC}$ est aussi l'adh\'erence de $\FC$ dans 
$\LC$. En utilisant alors l'hom\'eomorphisme 
$\s_{E_1}^* : \positif(\L/\L(E_1)) \stackrel{\sim}{\longrightarrow} \LC$ 
de la proposition \ref{topo espace} (g), 
on se ram\`ene \`a calculer l'adh\'erence de $\s_{E_1}^{* -1}(\FC)$ 
dans $\positif(\L/\L(E_1))$. Mais 
$$\s_{E_1}^{* -1}(\FC)=
\s_{E_1}^{* -1}(\UC_\L(E_2))=\UC_{\L/\L(E_1)}(\s_{E_1}(E_2)).$$
Or, d'apr\`es le corollaire \ref{adherence u}, on a 
$$\overline{\UC_{\L/\L(E_1)}(\s_{E_1}(E_2))} = 
\positif(\L/\L(E_1)) \setminus 
\bigl(\bigcup_{\l \in E_2} \UC_{\L/\L(E_1)}(\s_{E_1}(-\l))\bigr).$$
Par cons\'equent,
$$\overline{\FC}=\LC \cap \s_{E_1}^*\Bigl(\positif(\L/\L(E_1)) \setminus 
\bigl(\bigcup_{\l \in E_2} \UC_{\L/\L(E_1)}(\s_{E_1}(-\l))\bigr)\Bigr),$$
et donc
$$\overline{\FC}=\LC \cap 
\Bigl(\bigcap_{\l \in E_2} \bigl(\positif(\L)\setminus
\UC_\L(-\l)\bigr) = 
\bigcap_{\l \in E_1 \cup E_2} \overline{\DC_{\HC_\l}(X)},$$
comme attendu.

\medskip

Montrons maintenant (c). D'apr\`es (b), $\overline{\FC}$ est bien une r\'eunion 
de facettes. Si de plus $\FC'$ est une facette diff\'erente de $\FC$ et contenue 
dans $\overline{\FC}$, l'assertion (b) montre qu'il existe $\HC \in \AG_2$ 
tel que $\FC' \subseteq \HC$. Donc $\FC' \subseteq \LC_\AG(\FC) \cap \HC$, 
et $\dim \Pos^{-1}\bigl(\LC_\AG(\FC) \cap \HC\bigr) = 
\dim \Pos^{-1}\bigl(\LC_\AG(\FC)\bigr)-1$. D'o\`u (c). 

\medskip

L'assertion (d) d\'ecoule imm\'ediatement de (c). 
\end{proof}

\bigskip

On d\'efinit une relation $\infspe$ (ou $\preccurlyeq_\AG$ s'il est n\'ecessaire de 
pr\'eciser) entre les facettes~: on \'ecrit $\FC \infspe \FC'$ si 
$\overline{\FC} \subseteq \overline{\FC}'$ (c'est-\`a-dire si $\FC \subseteq \overline{\FC}'$). 
La proposition \ref{facettes} (d) montre que~:

\bigskip

\begin{coro}\label{ordre facettes}
La relation $\infspe$ entre les facettes est une relation d'ordre.
\end{coro}


\bigskip



\begin{thebibliography}{131}

\medskip

\bibitem{bonnafe} {\sc C. Bonnaf\'e}, Semi-continuity 
properties of Kazhdan-Lusztig cells, preprint (2008), available at 
{\tt arXiv:0808.3522}.

\medskip

\bibitem{bourbaki} {\sc N. Bourbaki}, {\it Groupes et alg\`ebres de Lie, 
chapitres IV, V et VI}, Hermann, Paris, 1968.

\medskip

\bibitem{robbiano} {\sc L. Robbiano}, 
Term orderings on the polynomial ring, EUROCAL '85, Vol. 2 (Linz, 1985),  513--517, {\it Lecture Notes in Comput. Sci.} {\bf 204}, Springer, 
Berlin, 1985.
\end{thebibliography}
\end{document}